\RequirePackage{fix-cm}
\documentclass[smallcondensed]{svjour3}     
\usepackage{amssymb}
\usepackage{amsmath}
\usepackage[dvips]{graphics}
\usepackage{color}
\usepackage[english]{babel}
\usepackage{pstricks-add}
\usepackage{graphicx}
\usepackage{pgf,tikz}
\usepackage{mathrsfs}
\usepackage{mathtools}
\usetikzlibrary{arrows}
\smartqed  
\usepackage{graphicx}

\newtheorem{conj}{Conjecture}
\newtheorem{Aff}{Affirmation}

\numberwithin{equation}{section}
\definecolor{zzttff}{rgb}{0.6,0.2,1.}
\definecolor{qqwuqq}{rgb}{0.,0.39215686274509803,0.}
\newrgbcolor{qqwuqq}{0. 0.39215686274509803 0.}
\newrgbcolor{zzttff}{0.6 0.2 1.}
\psset{xunit=1.0cm,yunit=1.0cm,algebraic=true,dimen=middle,dotstyle=o,dotsize=3pt 0,linewidth=0.8pt,arrowsize=3pt 2,arrowinset=0.25}
\begin{document}

\title{The Inverse Problem for Nested Polygonal Relative Equilibria}
%



\author{Marcelo P. Santos
}


\institute{Universidade Federal Rural de Pernambuco \\
     Depto de Matem\'atica\\
     Rua Dom Manoel de Medeiros, S/N,
      Dois Irm\~aos
     Recife, PE\\
     52171-900, Brasil\\
     Phone: 3320-6482, 
              \email{marcelo.pedrosantos@ufrpe.br}
              ORCID ID: orcid.org/0000-0002-9023-0728}           

\date{Received: date / Accepted: date}

\maketitle

\begin{abstract}
We prove that for some potentials (including the Newtonian one, and the potential of Helmholtz vortices in the plane) relative equilibria consisting of two homothetic regular polygons of arbitrary size can only occur if the masses at each polygon are equal. The same result is true for many ragular polygons as long as the ratio between the radii of the polygons are sufficient large. Moreover, under these hypotheses, the relative equilibrium always exist.
\keywords{Celestial Mechanics \and N-Body Problem \and N-Vortex Problem \and Central Configurations\and Relative Equilibrium\and Polygonal Central Configuration}
\subclass{70F10,70F15,70F17,70Fxx,37N05.}
\end{abstract}

\section{Introduction}
\label{sec:1}
The N-body Problem describe the dynamics of point masses under the action of gravitational law of attraction. Let $m_i$  represent point masses at positions $q_i \in \mathbb{R}^d$. The equations of motion are:
\begin{equation*}
m_i\ddot q_i=\left(\frac{\partial U_a}{\partial q_i}\right)^t,\quad i=1,\ldots, N,
\end{equation*}\\
where $t$ denotes transpose and
\begin{align}\label{PotencialQualquer}
U_a(q_1, \ldots, q_N)=\left\lbrace\begin{array}{c}
\frac{1}{a-2}\displaystyle\sum_{j<i} \frac{m_im_j}{|q_i-q_j|^{a-2}} \quad \mbox{ if } a > 2 ,\\
\displaystyle\sum_{j<i} m_im_j  \log |q_i-q_j| \quad \mbox{ if } a=2.\end{array}\right.
\end{align}

When $a=3$ we have the Newtonian case. This case is hard enough to prevent complete solutions if $N>2$, in fact, there is just one kind of solutions known explicitly: the homographic solutions, i.e., solutions whose shape is preserved along the motion up to scaling. In a homographic solution the initial conditions must satisfy the algebraic equations in the following definition.

\begin{definition}
A configuration $q=(q_1,\ldots,q_N)\in\mathbb{R}^{dN}$ is a {\bf central configuration } if there exists a constant $\gamma \in \mathbb{R}$ such that
\begin{equation}
\gamma(q_i-{\bf q_G})=\sum_{j \neq i} m_j\frac{q_j-q_i}{|q_j-q_i|^a} \quad i=1,\ldots, N,\label{EquacaoDeConfiguracaoCentral}
\end{equation}
where $ {\bf q_G}= \displaystyle \frac{\sum_{j=1}^{N} m_jq_j}{M}$ is the center of mass, and  $ \displaystyle M=\sum_{j=1}^{N} m_j \neq 0,$ is the total mass. 
\end{definition}

Central configurations have been of great interest over the last decades due their importance in Celestial Mechanics, for a good introduction see \cite{18,22}. 

  From now on we make the assumption $d=2$. In this case the central configuration is also called {\bf relative equilibrium} since in a rotating system of coordinates with angular velocity $\nu=\sqrt{-\gamma}$, a central configuration gives rise to an equilibrium solution of the N-body Problem.
  
   The Helmholtz's \cite{Helm} equations for the dynamics of $N$ point vertices in a planar
incompressible fluid with zero viscosity are given by:
\begin{equation}
\label{equacaoDeHelmholtz}
\dot q_i=\sqrt{-1}\cdot\sum_{j \neq i} m_j\frac{q_i-q_j}{|q_i-q_j|^2}.
\end{equation}

Here $m_j$ represents the vorticities which may either be positive or negative. 
 A relative equilibrium, in this case, corresponds to a periodic solution where the distances between vortices remain constant. Such equilibrium is a solution of equation \eqref{EquacaoDeConfiguracaoCentral} with $a=2$. For a deeper discussion of $N$-Vortex Problem we refer the reader to \cite{75}.
 
Among the most important problems in Celestial Mechanics, there is a problem of given positives masses finding all positions that gives rise to a central configuration (see for instance: \cite{1,74}). We describe the inverse problem by fix positions and find the masses which make it into a central configuration.

The aim of this paper is to study this problem when the positions are nested regular polygons, and the potential is like in  \eqref{PotencialQualquer}. In the remainder of this paper we assume that $N\geq 3$ and the polygons have no twisted angle.\\

It is easy to check that when the $N$ equal masses are located at
the vertices of a regular polygon, they form a relative equilibrium for a suitable choice of angular velocity. A relative equilibrium with positions in the  vertices of a regular polygon, with $N>3$, can only occur if the masses (vorticities) are equal, as shown in \cite{14,13} in the Newtonian case and \cite{10} in the $N$-Vortex case.

In \cite{8} this problem is discussed for two polygons, but  without a detailed proof. In \cite{21,2015} is shown in the Newtonian case that, for two polygons, a necessary condition for the existence of a relative equilibrium is that masses must be equal in each polygon. 

In this work we extend the previous result to the potentials \eqref{PotencialQualquer}, and we show that the conclusions remain true for such potentials when we have more polygons since the rate between their sizes is sufficient large.

 We cannot find in the literature any work on the inverse problem for more than two polygons in the Newtonian case. Or more than one polygon in the $N$-Vortex case. Here we examine theses cases.

\section{Relative Equilibrium of Nested Polygons}
\label{sec:2}
Consider $L$ regular concentric polygons where the angle between
 the polygons is zero. Consider $LN$ punctual masses $m_1,\ldots,m_{LN}$ at their vertices.

Let us assume that polygons are inscribed in circumferences of radii
$r_1, \ldots,r_L$ where $r_i \in \mathbb{R^+}$ and $ r_i \neq r_j$ if $i \neq j$ (in this way we will refer to the radius of the circumference which the polygon is inscribed).

To shorten notation, we write $I_k=\{1,2,3,\ldots,k\}.$ Enumerate the masses in a way that the first $N$ masses are at the polygon of radius $r_1$, the masses $m_{N+1}, \ldots,m_{2N}$ are at the polygon of radius $r_2$ and so on.

This configuration is a relative equilibrium, with angular velocity $\nu$ if and only if the following equation holds:
\begin{equation}
\nu^2(q_i-{\bf q_G})=\sum_{j=1\atop{j \neq i}}^{LN} m_j\frac{q_i-q_j}{|q_i-q_j|^a}, \quad \quad i\in I_{LN}.
\label{EQUACAODECONFIGURACAOCENTRAL}
\end{equation}

By identifying $\mathbb{R}^2\simeq \mathbb{C},$ we can write $q_{(j-1)N+k}=r_j\omega_k$
where $\omega_k=e^{\theta k \sqrt{-1}}$, with $\theta=\frac{2\pi}{N}$ and $j\in I_L ,k\in I_N$.
In this case, writing the equation for the body at {\it k-th} vertice of the {\it T-th} polygon , and indexing the polygons by $S$, equation \eqref{EQUACAODECONFIGURACAOCENTRAL} becomes 
 \begin{align*}
 \label{ggf}
 \nu^2(r_T\omega_k-{\bf q_G})=\sum_{S=1}^{L}\sum_{j=1\atop{j \neq k \,\text{\tiny if}\, S=T}}^{N} m_{(S-1)N+j} \frac{r_T\omega_k-r_S\omega_j}{|r_T\omega_k-r_S\omega_j|^{a}},
\end{align*}
for $T\in I_L, k=(T-1)N+1,\ldots, TN$.

Multiplying the {\it k-th} equation by $\omega_{-k}=e^{-\theta k \sqrt{-1}},$ and recalling the expression for ${\bf q_G}$, after some manipulations we get the equivalent equation
\begin{equation}
\label{2}
\nu^2r_T=\sum_{S=1}^{L}\left(\sum_{j=1\atop{j \neq k \,\text{\tiny if}\, S=T}}^{N} m_{(S-1)N+j} \frac{r_T-r_S\omega_{j-k}}{|r_T-r_S\omega_{j-k}|^{a}}+\frac{\nu^2}{M}\sum_{j=1}^{N} m_{(S-1)N+j} r_S\omega_{j-k}\right).
\end{equation}

 Now we will define the matrices $ A_ {TS}=[a_{kj}]$, that represents the interaction
 between the bodies present in the polygons indexed by $T$ and $S$ respectively.
  More precisely, $a_{kj}$ express the interaction 
 between the {\it k-th} body at {\it T-th} polygon and {\it j-th} body at {\it S-th} polygon.

Consider the $L^2$ matrices of order $N\times N$, $A_{TS}=[a_{kj}]$ , where, for $k,j\in I_{N}$, we let 

\begin{equation}
  a_{kj}=\left\{\begin{array}{cr}
  \dfrac{\nu^2}{M}r_T ,&\text{ if } T=S \text{ and } k=j,\\
  \dfrac{r_T-r_S\omega_{j-k}}{|r_T-r_S\omega_{j-k}|^{a}}+\frac{\nu^2}{M}r_S\omega_{j-k}, &\text{ otherwise.}  \\
    \end{array}\right.\label{MatrizATS}
\end{equation}
Equations (\ref{2}) become \begin{equation}
 \left[\begin{array}{cccc}
  A_{11}&A_{12}&\ldots&A_{1L}\\
  A_{21}&A_{22}&\ldots&A_{2L}\\
  \vdots&\vdots&\ddots&\vdots\\
  A_{L1}&A_{L2}&\ldots&A_{LL}\\
  \end{array}\right]
   \left[\begin{array}{c}
  \mathbf m_{1}\\
  \mathbf m_{2}\\
    \vdots\\
 \mathbf m_{L}
  \end{array}\right]=\left[\begin{array}{c}
  \nu^2r_1 \mathbf{1}\\
  \nu^2r_2\mathbf{1}\\
    \vdots\\
 \nu^2r_L\mathbf{1}\\
  \end{array}\right]\label{EquacaoMatricialParaOProblema},
  \end{equation}
where $\mathbf{m}_T=\left(m_{(T-1)N+1},\ldots, m_{TN}\right)^t,
\mathbf{1}= \left(1,\ldots,1\right)^t \in \mathbb{R}^N$ and $t$ denotes the transpose. 

So the aim is to find the values of $\nu$ and masses such that system \eqref{EquacaoMatricialParaOProblema} has solution.

The matrices $A_{TS}$ are  simultaneously diagonalizable (see section \ref{sec:3}) by a basis of eigenvectors $\{\mathbf{v}_p\}_{p\in I_N}$, such that $\mathbf{v}_N=\mathbf{1}$. 

If we set $\lambda_p(A_{TS})$ the eigenvalue associated to $\mathbf{v_p}$, and $\otimes$ denotes the Kronecker product of matrices, and $\{\mathbf{e}_j\}_{j \in I_L}$ the canonical basis of $\mathbb{C}^L$. Recalling that $\{\mathbf{v}_p\}_{p\in I_N}$  is a basis of $\mathbb{C}^N$, we see that the vector 
\begin{equation}
\label{VetorDeMassas}
\mathbf{m}=(\mathbf{m}_1,\ldots,\mathbf{m}_L)=\sum_{p=1}^{N} (x_1^p\mathbf{e}_1\otimes\mathbf{v}_p+\ldots+x_L^p\mathbf{e}_L\otimes\mathbf{v}_p),
\end{equation}

 is a solution of system \eqref{EquacaoMatricialParaOProblema}
if, and only if, the coefficients $x_i^p$, for $p\in I_{N-1},i\in I_{L}$,  satisfy the subsystem:
\begin{equation}
\left\{\begin{array}{ccc}
x_1^p\lambda_p(A_{11})+x_2^p\lambda_p(A_{12})+\ldots+x_L^p\lambda_p(A_{1L})&=&0,\\
x_1^p\lambda_p(A_{21})+x_2^p\lambda_p(A_{22})+\ldots+x_L^p\lambda_p(A_{2L})&=&0,\\
\vdots&\vdots& \\
x_1^p\lambda_p(A_{L1})+x_2^p\lambda_p(A_{L2})+\ldots+x_L^p\lambda_p(A_{LL})&=&0,
\end{array}\right.\label{SubsistemaP}
\end{equation}
and for $p=N,$ the subsystem:
\begin{equation}
\left\{\begin{array}{ccc}
x_1^N\lambda_N(A_{11})+x_2^N\lambda_N(A_{12})+\ldots+x_L^N\lambda_N(A_{1L})&=&\nu^2r_1,\\
x_1^N\lambda_N(A_{21})+x_2^N\lambda_N(A_{22})+\ldots+x_L^N\lambda_N(A_{2L})&=&\nu^2r_2,\\
\vdots&\vdots& \\
x_1^N\lambda_N(A_{L1})+x_2^N\lambda_N(A_{L2})+\ldots+x_L^N\lambda_N(A_{LL})&=&\nu^2r_L.
\end{array}\right.\label{SubsistemaN}
\end{equation}

Since the coefficients $x_i^p$ determine the mass vector  $\mathbf{m}$, we can reformulate the inverse problem by considering the $x_i^p$ and $\nu$ as the unknowns once the positions and the numbers $\lambda_p(A_{TS})$ are given.

We will show under certain hypothesis that the determinant $\det[\lambda_p(A_{TS})],\, T,S \in I_L$ has to be different from zero, so the corresponding coefficients $x_i^p$ are zero. Also we show that $p=N$, $\det[\lambda_N(A_{TS})],T,S\in I_{L}$ is non-zero. This establishes the existence of masses such that  equations \eqref{EquacaoMatricialParaOProblema} hold. Furthermore, if $\nu$ is real, this implies that the masses must be real too.

\section{Circulant Matrices and The Determinant Expression}
\label{sec:3}
\begin{definition}
A matrix $C$ of order $N\times N,$ is circulant if $c_{i-1,j-1}=c_{i,j} $ for all $i,j\in I_N$,
where we identify $c_{0,j}$ with $c_{N,j}$,  and $c_{i,0}$ with $c_{i,N}.$
\end{definition}
It follows immediately that matrices defined in \eqref{MatrizATS} are circulant.

\begin{lemma}
\label{LemaOfCiculantsMatrices}
Any circulant matrix $C=[c_{ij}]$, is given by the polynomial\\
$C=\displaystyle \sum_{j=1}^{N} c_{1,j}W^{N-j+1}$, where $\displaystyle W \in M_{N\times N}(\mathbb{C})$
is the circulant matrix:
\begin{equation*}
W=
\left(\begin{array}{ccccc}
0&0&\ldots&0&1\\
1&0&\ldots&0&0\\
\vdots&\ddots&\ddots&\vdots&\vdots\\
\vdots&\vdots&\ddots&\ddots&\vdots\\
0&\ldots&0&1&0
\end{array}\right).
\end{equation*}
The set of vectors $\{\mathbf{v}_p\}_{p\in I_{N}}$, with $\mathbf{v}_p=\left(1,\omega_p,\cdots,\omega_p^{N-1}\right)^t$,with $\omega_p=e^{\theta p\sqrt{-1}},\theta=\frac{2\pi}{N}$,  is a eigenvectors basis for $W$, and the eigenvalue
associated to $\mathbf{v}_p$ is $\lambda_p(W)=\omega_p^{N-1}$. Hence, $C$ is diagonalizable with the same basis and correspondent eigenvalues are  
\begin{equation}
\sum_{j=1}^{N} c_{1,j}(\omega_{p}^{N-1})^{N-j+1}=\sum_{j=1}^{N} c_{1,j}\omega_{p}^{j-1}.
\label{FormulaAutovalorDaMatrizCirculante}
\end{equation}
\end{lemma}
\begin{proof}
A direct calculation shows that the characteristic polynomial of $W$ is given by $P^{W}(x)=x^{N}-1$.
Thus the roots of unity forms a full set of eigenvalues of $W$. Moreover, its easy to see that $W(\mathbf{v}_p)=\omega_p^{N-1}\cdot \mathbf{v}_p$. The expression of $C$, as polynomial, follows from the form of the powers for $W$. From this expression we conclude that $C$ share the same eigenbasis of $W$, and the eigenvalues for $C$ are the correspondent polynomial calculated at eigenvalues of $W$.
\end{proof}
For more, properties of circulants matrices, we refer the reader to \cite{34}.

Using Lemma \ref{LemaOfCiculantsMatrices} we easily establish the following lemma.
\begin{lemma}
\label{LemaFormaComplexaDosAutovalores}
The eigenvalue $\lambda_p(A_{TS})$ associated to the eigenvector $\mathbf{v}_p$ of $A_{TS}$ is determined by the expression:
\begin{align}
\begin{array}{ccl}
\label{AutovalorATSformacomplexaPDifN-1}
\lambda_p(A_{TS})&=&\displaystyle \sum_{j=1}^{N} \dfrac{r_T-r_S\omega_{j-1}}{|r_T-r_S\omega_{j-1}|^{a}}\omega_{p}^{j-1}+\delta_{p,N-1}\cdot r_S\frac{\nu^2}{M}N,\quad \mbox{ if }\quad T \neq S,\\
\lambda_p(A_{TT})&=&\displaystyle \sum_{j=2}^{N} \dfrac{r_T-r_T\omega_{j-1}}{|r_T-r_T\omega_{j-1}|^{a}}\omega_{p}^{j-1}+\delta_{p,N-1}\cdot r_T\frac{\nu^2}{M}N,
\end{array}
\end{align}
where $\delta_{p,N-1}$ is the Kronecker delta.
\end{lemma}

\begin{proof}
Consider the case where $T \neq S.$ Using \eqref{MatrizATS} and \eqref{FormulaAutovalorDaMatrizCirculante} the eigenvalues have expressions:
\begin{equation*}
\lambda_p(A_{TS})=\sum_{j=1}^{N}\left( \frac{r_T-r_S \omega_{j-1}}{|r_T-r_S\omega_{j-1}|^{a}}+\frac{\nu^2}{M}r_S\omega_{j-1}\right)(\omega_p)^{j-1},\quad p\in I_{N}.\label{formuladeautovalor}
\end{equation*}
The proof is completed by noticing that $\displaystyle \sum_{j=1}^{N}\omega_{j-1}\omega_p^{j-1}=\left\{\begin{array}{ccc}
0 & if & p \neq N-1\\
N & if & p=N-1
\end{array}\right.$.
The case  $T=S$  is treated likewise. It is worth noting that only $\lambda_{N-1}(A_{TS})$ depends on $\frac{\nu^2}{M}$.
\end{proof}

 The following lemma is based on Lemmas $5$ and $9$ of \cite{13}.
\begin{lemma}
\label{LemaDosAutoValoresReais}
If $x$ and $y$ are real, the sums $ \displaystyle \lambda_p(x,y)=\sum_{j=1}^{N} \frac{x-y\omega_{j-1}}{|x-y\omega_{j-1}|^{a}}\omega_p^{j-1}$ and \\ $\tilde{\lambda}_p(x,y)= \displaystyle \sum_{j=2}^{N} \frac{x-y\omega_{j-1}}{|x-y\omega_{j-1}|^{a}}\omega_p^{j-1}$ are real.
\end{lemma}

\begin{proof}
In the expression $\lambda_p(x,y)$ it suffices to note that the first term is real and if $j\geq 2$ the $j$-th term is the complex conjugate of the $(N-j+2)$-th. Thus  $\lambda_p(x,y)$ is real, and so is $\tilde{\lambda}(x,y)$.
\end{proof}

Applying Lemma \ref{LemaDosAutoValoresReais} to the expressions of the eigenvalues in \eqref{AutovalorATSformacomplexaPDifN-1} and making some simplifications, these expressions become respectively the $f_p$ and $\xi_p$ given below.

\begin{definition}
Let $p \in I_N$, and $\theta=\frac{2\pi}{N}$, we define $f_{p}: \mathbb{R}^2\setminus \left\{(x,x)| x \in \mathbb{R}\right\} \rightarrow \mathbb {R}$ by
\begin{equation}
f_p(r_T,r_S)=\sum_{j=1}^{N} \frac{r_T\cos\left( j\theta p\right)-r_S\cos\left( j\theta (p+1)\right)}{\left(r_T^2-2r_Tr_S\cos\left( j\theta\right)+r_S^2\right)^{\frac{a}{2}}},\label{FuncaoF}
\end{equation}
and $\xi_p:\mathbb{R}^{+}\rightarrow \mathbb{R}$ by
\begin{equation}
\xi_p(r_T)=(2r_T)^{1-a}\sum_{j=1}^{N-1}  \sin\left(\frac{ j\theta (2p+1)}{2}\right)\sin^{1-a}\left(\frac{ j\theta}{2}\right).\label{FuncaoE}
\end{equation}
\end{definition}

By the Lemmas \ref{LemaFormaComplexaDosAutovalores} and \ref{LemaDosAutoValoresReais}, $\det[\lambda_p(A_{TS})]$ is a function of the radii $r_i,i\in I_{L}$ and if $p \neq N-1$, could be expressed by:
\begin{equation*}
\det[\lambda_p(A_{TS})]
=\left|\begin{array}{cccc}
\xi_p(r_1)&f_p(r_1,r_2)&\ldots&f_p(r_1,r_L)\\
f_p(r_2,r_1)&\xi_p(r_2)&\ldots&f_p(r_2,r_L)\\
\vdots&\vdots&\ddots&\vdots\\
f_p(r_L,r_1)&f_p(r_L,r_2)&\ldots&\xi_p(r_L)
\end{array}\right| .
\end{equation*}
Since $f_p$ and $\xi_p$ are homogeneous functions of degree $1-a$, we  conclude that $\det[\lambda_p(A_{TS})]$ is product of $(r_1\dotsm r_L)^{1-a}$ by the factor
\begin{align}
\left|\begin{array}{cccc}
\xi_p(1)&f_p\left(1,\frac{r_2}{r_1}\right)&\ldots&f_p\left(1,\frac{r_L}{r_1}\right)\\
f_p\left(1,\frac{r_1}{r_2}\right)&\ddots&\ldots&\vdots\\
\vdots&\ldots&\ddots&f_p\left(1,\frac{r_{L}}{r_{L-1}}\right)\\
f_p\left(1,\frac{r_1}{r_L}\right)&\ldots&f\left(1,\frac{r_{L-1}}{r_L}\right)&\xi_p(1)
\end{array}\right|.
\label{DefinicaoDafuncaoGammaIgualAoDet}
\end{align}
Thus $\det[\lambda_p(A_{TS})]$ is nonzero as long as the above determinant is nonzero. 

\section{Two polygons}
\label{Sec:3}

For two polygons \eqref{DefinicaoDafuncaoGammaIgualAoDet} becomes
\begin{align}
\left|\begin{array}{cc}
\xi_p(1)&f_p\left(1,\frac{r_2}{r_1}\right)\\f_p\left(1,\frac{r_1}{r_2}\right)&\xi_p(1)
\end{array}\right|.
\label{Determinante2Por2DoSubsistemaP}
\end{align}

 In order to show  that determinant is strictly positive, we only need to show that $f_p(1,x)$ and $f_p\left(1,\frac{1}{x}\right)$ have  opposite signs and are different from zero, here $x=\frac{r_2}{r_1} \neq 1$.
By homogeneity $f_p\left(1,\frac{1}{x}\right)=x^{a-1}f_p(x,1)$, so it is enough to show that $f_p(1,x)$ e $f_p(x,1)$ have different signs.
But $f_{N-p-1}(1,x)=-f_p(x,1)$ and $f_{N+p}(x,y)=f_p(x,y)$, hence   
it suffices to prove
 that $f_p(x,1)$ has always the same sign for all $p\in I_N$.

In the subsection \ref{subsec:4.1} we make this analysis inspired by an idea used in lemma $2$ of \cite{5}. Once again by symmetry, if we prove the particular case $x \in (0,1)$, the general assertion follows, once for $x>1$ we could use the identities $f_p\left(1,\frac{1}{x}\right)=x^{a-1}f_p(x,1)$ and $f_p\left(\frac{1}{x},1\right)=x^{a-1}f_p(1,x)$.

\subsection{Analysis of the sign function $f_p(x,1)$}
\label{subsec:4.1}

We want to show that the function $f_p(x,1)$ is negative if $x$ is in the open interval $(0,1)$. To that effect, we will find an explicit power series for $f_p(x,1)$ whose all coefficients are all non-positive. This implies that $f_p(x,1)$ and all its derivatives are negative in $(0,1)$.

    Consider the function  $\phi: \mathbb{C}\setminus\{1\}\rightarrow \mathbb{C}$  given by
$\phi(z)=1/(1-z)^{\frac{a}{2}}$. Its Taylor series around
 the origin
$\phi(z)=\displaystyle \sum_{n=0}^{\infty} \alpha_n z^{n}$ converges at open unit disc and
\begin{Aff}
 If $a$ is positive then all coefficients $\alpha_k$ are positive.
\end{Aff}
\begin{Aff}
 If $a> 2$ the sequence of coefficients $\alpha_k$ is increasing, and if $a=2$ the sequence is constant.
\end{Aff}

To prove these affirmations we only need to note  $\displaystyle \alpha_k=(-1)^k\binom{
-\frac{a}{2}}{k}$, where the last parenthesis stands for binomial coefficients.

Consider $x \in \mathbb{R}$ with $|x|<1$ , so
\begin{align}
\frac{1}{(1-2x\cos(\theta j)+x^2)^{\frac{a}{2}}}&=&\frac{1}{\left(1-xe^{+\theta j\sqrt{-1} }\right)^{\frac{a}{2}}}\cdot\frac{1}{\left(1-xe^{-\theta  j\sqrt{-1}}\right)^{\frac{a}{2}}}\nonumber\\
&=&
 \sum_{n=0}^{\infty}\left( \sum_{k+l=n} \alpha_k\alpha_l e^{\theta j(k-l)\sqrt{-1}}\right)x^{n}.\label{identidadeDaMultiplicacaodeSeries}
\end{align}
Using (\ref{identidadeDaMultiplicacaodeSeries}) we write for $x \in (0,1)$:

\begin{align*}
&f_p(x,1)=\sum_{j=1}^{N} \frac{x\cos\left( j\theta p\right)-\cos\left( j \theta(p+1)\right)}{\left(1-2x\cos\left( j\theta\right)+x^2\right)^{\frac{a}{2}}}=\displaystyle\sum_{n=0}^{\infty}\beta_n x^{n},
\end{align*}
where
\begin{align*}
&\beta_0=-(\alpha_0)^2\left( \sum_{j=1}^{N} \cos\left(j\theta (p+1)\right)\right),\\
\displaystyle\beta_n=&\sum_{k+l=n-1} \alpha_k\alpha_l \sum_{j=1}^{N} \cos\left( j \theta p\right) e^{\theta j(k-l)\sqrt{-1}}\\
&-\sum_{k+l=n} \alpha_k\alpha_l\sum_{j=1}^{N} \cos\left( j\theta(p+1)\right) e^{\theta j(k-l)\sqrt{-1}}.
\end{align*}

Consider only the coefficients $\beta_n$ from powers of degree non-zero, by isolating the index $l$ we  get:
\begin{align}
\begin{split}
\beta_n=
-\alpha_n\alpha_0\left( \sum_{j=1}^{N} \cos\left( j\theta(p+1)\right) e^{\theta j n\sqrt{-1}}\right)+\end{split}\nonumber\\
\begin{split}
\sum_{k=0}^{n-1} \alpha_k\left[\alpha_{n-k-1}\sum_{j=1}^{N} \cos\left( j \theta p\right) e^{\theta j(2k-(n-1))\sqrt{-1}}\right.\\
\left.- \alpha_{n-k}\sum_{j=1}^{N} \cos\left( j\theta(p+1)\right) e^{\theta j(2k-n)\sqrt{-1}}\right] .
\end{split}\label{galapagos}
\end{align}
It is easy to check that for integers $u$ and $v$ the imaginary part of\\ $\displaystyle\sum_{j=1}^{N} \cos\left( j\theta u\right) e^{\theta jv\sqrt{-1}}$ vanishes, so
\begin{align}
&\sum_{j=1}^{N} \cos\left( j\theta u\right) e^{\theta jv\sqrt{-1}}=\frac{1}{2}\sum_{j=1}^{N}\left[ \cos\left( j\theta(u+v)\right)+ \cos\left( j\theta(u-v)\right)\right]
\label{summersong}.
\end{align}
Using \eqref{summersong} the expression for $\beta_n$ in \eqref{galapagos} becomes half of the following sum:
\begin{align}
&-\alpha_n\alpha_0\left( \sum_{j=1}^{N} \cos\left( j\theta(p+1+n)\right)+\cos\left( j\theta(p+1-n)\right) \right)+\label{CauseWeveEndedAsLovers}\\
&\sum_{k=0}^{n-1} \alpha_k\left[(\alpha_{n-k-1}-\alpha_{n-k})\left(\sum_{j=1}^{N} \cos\left( j\theta(p+1+(2k-n))\right)\right)\right]+\label{europa}\\
&\sum_{k=0}^{n-1}\alpha_{k}\left[\alpha_{n-k-1}\sum_{j=1}^{N}\cos\left( j\theta(p-2k+n-1)\right)-\alpha_{n-k}\sum_{j=1}^{N}\cos\left( j\theta(p+1-2k+n)\right)\right].
\label{blackstar}
\end{align}
We shall now prove that 
the sum above is negative or zero. For this note that expressions in  \eqref{CauseWeveEndedAsLovers} and \eqref{europa} are already negative or zero because
for integers $s$ the sum $\sum_{j=1}^{N} \cos\left( j\theta s\right)$ is $N$ if $s$ is a multiple of $N$ and $0$ otherwise, additionally
the sequence of positive terms $\alpha_k$ is increasing.\\
 The expression in \eqref{blackstar} only contain a $k$-th term positive if 
 \begin{equation}
 \sum_{j=1}^{N}\cos\left( j\theta(p-2k+n-1)\right)= N \quad \text{and} \quad \sum_{j=1}^{N}\cos\left( j\theta(p+1-2k+n)\right)=0.
 \label{layla}
 \end{equation}
 We will now show that in this case there is always a negative one to compensate.
If $k\neq n-1$, consider the sum of terms of index $k$ and $(k+1)$ in \eqref{blackstar}
\begin{align*}
&\alpha_{k}\left[\alpha_{n-k-1}\sum_{j=1}^{N}\cos\left( j\theta(p-2k+n-1)\right)-\alpha_{n-k}\sum_{j=1}^{N}\cos\left( j\theta(p+1-2k+n)\right)\right]+\\
&\alpha_{k+1}\left[\alpha_{n-k-2}\sum_{j=1}^{N}\cos\left( j\theta(p-2k+n-3)\right)-\alpha_{n-k-1}\sum_{j=1}^{N}\cos\left( j\theta(p-2k+n-1)\right)\right].
\end{align*}

 By \eqref{layla}, the number $p-2k+n-1$ is multiple of $N$, since  and $N\geq 3$, then $p-2k+n-3$ is not multiple of $N$, so the sum $\sum_{j=1}^{N}\cos\left(j\theta(p-2k+n-3)\right)$ is zero. Then adding the $k$-th term with the $(k+1)$-th term we get
\begin{align*}
&\alpha_k\alpha_{n-k-1}N+\alpha_{k+1}(-\alpha_{n-k-1}N)=N\alpha_{n-k-1}(\alpha_k-\alpha_{k+1}),
\end{align*}
which is negative because the sequence of positive terms $\alpha_k$ is increasing.

If $k=n-1$, we have $\sum_{j=1}^{N}\cos\left( j\theta(p-2k+n-1)\right)=\sum_{j=1}^{N}\cos\left( j\theta(p+1-n)\right)=N$ and the $k$-th term it would be $\alpha_{n-1}\alpha_0N$. Note that a factor in \eqref{CauseWeveEndedAsLovers} have the term $\sum_{j=1}^{N}\cos\left( j\theta(p+1-n)\right)=N$ as well. The sum of both would result in $N(\alpha_{n-1}\alpha_0-\alpha_n\alpha_{0})=N\alpha_{0}(\alpha_{n-1}-\alpha_n)\leq 0$.
This finishes the proof.
\subsection{Conclusion for two polygons case}

The determinant \eqref{Determinante2Por2DoSubsistemaP} it will be strictly positive if $p \neq N-1$, so the corresponding coefficients in \eqref{VetorDeMassas} are zero. This show that a possible solution for the equations (\ref{EquacaoMatricialParaOProblema}) with $L=2$, must be given by:
\begin{equation}
\mathbf{m}=\left(x_1^{N-1}\mathbf{e}_1\otimes\mathbf{v}_{N-1}+x_2^{N-1}\mathbf{e}_2\otimes\mathbf{v}_{N-1}+x_1^{N}\mathbf{e}_1\otimes\mathbf{v}_{N}+x_2^{N}\mathbf{e}_2\otimes\mathbf{v}_{N}\right).
\end{equation}

It is not possible that all the coordinates of the vector $x_1^{N-1}\mathbf{e}_1\otimes\mathbf{v}_{N-1}+x_2^{N-1}\mathbf{e}_2\otimes\mathbf{v}_{N-1}$  are real if $N>2$ (see the lemma 6,  in \cite{13},\footnote{Note that  $\mathbf{v}_{N-1}=\omega_{N-1}\mathbf{V}_{N}$, where $\mathbf{V}_{N}$ is the vector on the referred lemma, and $\tilde{a}\mathbf{V}_{N}\in \mathbb{R}^N$ only if $\tilde{a}=0$ or $N=2$.}).
Hence for real masses (or vorticities) the vector solution must be 
$x_1^N\mathbf{e}_1\otimes\mathbf{v}_{N}+x_2^{N}\mathbf{e}_2\otimes\mathbf{v}_{N}$
with $x_1^{N}$ and $x_2^{N}$ being real numbers.

 We conclude that the masses in each polygon must be equal. Additionally, we note as consequence of determinant be non-zero for the case $P=N,$ it follows that for any fixed positive radii $r_1$ and $r_2$ and an arbitrary angular velocity $\nu $, there exists masses (equal in each polygon), which make a relative equilibrium.
And this masses are uniquely determined. Although occasionally are not positive as we will see in the Subsection \ref{sign_of_masses}.

We summarize the results obtained in the next theorem.
\begin{theorem}
\label{TwoPolygons}
A relative equilibrium consisting of two homothetic regular polygons, associated to a potential of the form (\ref{PotencialQualquer}), only is possible if the masses in each polygon are equal.
Moreover fixed any radii $r_1\neq r_2$, this configuration aways exist.
\end{theorem}
 Its interesting to compare with the results in \cite{5}, where is demonstrated that for a choice of masses (equal in each polygon), there exists two pair of radius which make it the configuration a relative equilibrium.

\subsection{The sign of the masses in the central configuration}
\label{sign_of_masses}
The masses (or vorticities) that gives rise to a relative equilibrium are not always positive, in fact if the polygons are close, then the masses corresponding could not be positive.
More specifically, we demonstrated the following theorem:

\begin{theorem} Let $m_1$ be the mass at polygon of radius $r_1$ and $m_2$ be the mass at polygon of radius $r_2$, in a central configuration where the positions form homothetic regular polygons. Fixed the radii $r_2>r_1$ there is an open interval $(0,\delta)$ such that if $r_1 \in (0,\delta)$ the masses $m_1$ and $m_2$ must have same sign. And if $r_1\in (\delta, r_2)$ the masses have opposite signs.
\end{theorem}

\begin{proof}

By applying the Crammer's rule in the system \eqref{SubsistemaN}, with $L=2$, we get:
\begin{equation}
\label{EquacaoDasMassas}
m_1=\frac{\left|\begin{array}{cc}
r_1\nu^2 & f_N(r_1,r_2)\\
r_2\nu^2 & \xi_N(r_2)
\end{array}\right|}{\left|\begin{array}{cc}
\xi_N(r_1)& f_N(r_1,r_2)\\
f_N(r_2,r_1)& \xi_N(r_2)
\end{array}\right|}\quad\text{ and }\quad
m_2=\frac{\left|\begin{array}{cc}
\xi_N(r_1)& r_1\nu^2\\
f_N(r_2,r_1)& r_2\nu^2
\end{array}\right|}{\left|\begin{array}{cc}
\xi_N(r_1)& f_N(r_1,r_2)\\
f_N(r_2,r_1)& \xi_N(r_2)
\end{array}\right|}
\end{equation}
We have proved in the Subsection \ref{subsec:4.1}
Since the denominator is positive, the sign of the masses will be given by the numerator.
Thus the sign of $m_1$ is equal to the sign of $\nu^2$ because $\left(r_1\xi_N(r_2)-r_2f_N(r_1,r_2)\right)>0,$ since
$$\xi_N(r)=(2r)^{1-a}\sum_{j=1}^{N-1}  \sin^{2-a}\left(\frac{\pi j}{N}\right)>0,\text{ and  } f_N(r_1,r_2)=(r_2)^{1-a}f_N\left(\frac{r_1}{r_2},1\right)<0.$$

The mass $m_2$ have the same sign of $\nu^2(\xi_N(r_1)r_2-r_1f_N(r_2,r_1))$. From \eqref{FuncaoF} we conclude that $\displaystyle\lim_{r_1 \to r_2^-} f_N(r_2,r_1)=+\infty$, and this implies that $$\lim_{r_1 \to r_2^-} \xi_N(r_1)r_2-r_1f_N(r_2,r_1)=-\infty.$$
Moreover, using again \eqref{FuncaoF} follows that $\displaystyle\lim_{r_1\to 0^{+}} f_{N}(r_2,r_1)=(r_2)^{1-a}N$ and
 by \eqref{FuncaoE} we have $\displaystyle\lim_{r_1\to 0^{+}} \xi_{N}(r_1)=+\infty$, then
\begin{equation*}
\lim_{r_1 \to 0^+} \xi_N(r_1)r_2-r_1f_N(r_2,r_1)=+\infty.
\end{equation*}

To finish the proof is sufficient to show that $\xi_N(r_1)r_2-r_1f_N(r_2,r_1)$ is a monotone decreasing function in $r_1$. To  see this, note that $\xi_N(r_1)$ is decreasing in $r_1$.
And $f_N(r_2,r_1)$ is increasing in $r_1$, because $f_N(r_2,r_1)=(r_2)^{1-a}f_N\left(1,\frac{r_1}{r_2}\right)$ so
\begin{equation*}
\frac{d}{dr_1}f_N(r_2,r_1)=(r_2)^{1-a}f'_N\left(1,\frac{r_1}{r_2}\right).\frac{1}{r_2},
\end{equation*}
and $f'_N\left(1,\frac{r_1}{r_2}\right)$ is positive because by the previous sections $f_{N-1}(x,1)$ and all its derivatives are negative in $(0,1)$, and $f_N(1,x)=-f_{N-1}(x,1)$. It follows that $-r_1f_{N}(r_2,r_1)$ is decreasing in $r_1$. This proves the theorem.
\end{proof}

\begin{remark}
It is interesting to note that the case of negative masses is
physically important if $a=2$ ($N$-vortex case).
\end{remark}
\section{More polygons}
\label{sec:4}
\begin{theorem}
\label{EvenPolygons}
Let $L$ be an even number. Set $a \geq 2$, and consider a configuration associated to a potential of the form (\ref{PotencialQualquer}). 
It is possible choose radii, with $r_3,r_4$ much smaller than $r_1$ and $r_2$, and $r_5,r_6$ much smaller than $r_3$ and $r_4$  and so on, such that for any choice of the angular velocity $\nu$ there is an relative equilibrium consisting of $L$ homothetic regular polygons with these radii. Moreover such equilibrium only is possible if the masses in each polygon are equal.
\end{theorem}
\begin{proof} 
We already know that theorem is valid if $L=2$. We now proceed by induction in $k$ where $L=2k$. By the discussion in Section \ref{sec:3},  only remains to prove that \eqref{DefinicaoDafuncaoGammaIgualAoDet} is non-zero.
We will show that if we choose $r_{2k-1}$ and $r_{2k}$ sufficiently small this determinant will be different from zero. In fact taking the limit when $\alpha \to 0^{+},$ in the following expression for $p \neq N-1,N$:
\begin{align*}
\left|\begin{array}{ccccccc}
\xi_p(1)&f_p\left(1,\frac{r_2}{r_1}\right)&\ldots & f_p\left(1,\frac{r_{2k-2}}{r_1}\right)&f_p\left(1,\frac{\alpha r_{2k-1}}{r_1}\right)&f_p\left(1,\frac{\alpha r_{2k}}{r_1}\right)\\
\vdots & \ddots & \vdots & \vdots & \vdots &\vdots \\
f_p\left(1,\frac{r_1}{r_{2k-2}}\right)&f_p\left(1,\frac{r_2}{r_{2k-2}}\right)&\ldots & \xi_p(1)&f_p\left(1,\frac{\alpha r_{2k-1}}{r_{2k-2}}\right)&f_p\left(1,\frac{\alpha r_{2k}}{r_{2k-2}}\right)\\
f_p\left(1,\frac{r_1}{\alpha r_{2k-1}}\right)&f_p\left(1,\frac{r_2}{\alpha r_{2k-1}}\right)&\ldots & f_p\left(1,\frac{r_{2k-1}}{\alpha r_{2k-1}}\right)&\xi_p(1)&f_p\left(1,\frac{\alpha r_{2k}}{ \alpha r_{2k-1}}\right)\\
f_p\left(1,\frac{r_1}{\alpha r_{2k}}\right)&f_p\left(1,\frac{r_2}{\alpha r_{2k}}\right)&\ldots & f_p\left(1,\frac{r_{2k-2}}{\alpha r_{2k}}\right)&f_p\left(1,\frac{\alpha r_{2k-1}}{\alpha r_{2k}}\right)&\xi_p(1)
\end{array}\right|\label{Determinantegrande}
\end{align*}
we obtain
\begin{align*}
\left|\begin{array}{ccccccc}
\xi(1)&f_p\left(1,\frac{r_2}{r_1}\right)&\ldots & f_p\left(1,\frac{r_{2k-2}}{r_1}\right)&0&0\\
\vdots & \ddots & \vdots & \vdots & \vdots &\vdots \\
f_p\left(1,\frac{r_1}{r_{2k-2}}\right)&f_p\left(1,\frac{r_2}{r_{2k-2}}\right)&\ldots & \xi_p(1)&0&0\\
0&0&\ldots & 0&\xi_p(1)&f_p\left(1,\frac{ r_{2k}}{ r_{2k-1}}\right)\\
0&0&\ldots & 0&f_p\left(1,\frac{ r_{2k-1}}{ r_{2k}}\right)&\xi_p(1)
\end{array}\right|
\end{align*}
By the induction hypothesis,
it is possible to choose radii whose determinant of the first block on the diagonal  is nonzero.  By the discussion in Section \ref{sec:3}, the $2\times 2$ block diagonal has a nonzero determinant.

If $p=N$ the limit results in
\begin{align*}
\left|\begin{array}{ccccccc}
\xi(1)&f_p\left(1,\frac{r_2}{r_1}\right)&\ldots & f_p\left(1,\frac{r_{2k-2}}{r_1}\right)&N&N\\
\vdots & \ddots & \vdots & \vdots & \vdots &\vdots \\
f_p\left(1,\frac{r_1}{r_{2k-2}}\right)&f_p\left(1,\frac{r_2}{r_{2k-2}}\right)&\ldots & \xi_p(1)&N&N\\
0&0&\ldots & 0&\xi(1)&f_p\left(1,\frac{ r_{2k}}{  r_{2k-1}}\right)\\
0&0&\ldots & 0&f_p\left(1,\frac{ r_{2k-1}}{ r_{2k}}\right)&\xi(1)
\end{array}\right|
\end{align*}

Again we concluded that such determinant is different from zero. By continuity if $\alpha$ is sufficient smaller the determinant is different from zero. Equivalently if $r_{2k-1}$ and $r_{2k}$ are sufficiently small the correspondent determinant is non-zero. 

Now we note that the determinant \eqref{DefinicaoDafuncaoGammaIgualAoDet} is invariant if we multiply all radii by the same factor, so we could modify the radii as long as we maintain their ratio constant. 
This proves the theorem.\end{proof}
\begin{remark}
Consider the last theorem, if $p=N-1$ we know that determinant of subsystem it is not given by  expression \eqref{DefinicaoDafuncaoGammaIgualAoDet}, however, the vectors from the correspondent subspace are not real, as mentioned earlier.
\end{remark}


\begin{theorem}
Let $L$ be a odd number. Let $a =2 $ or $a=3$, and consider the equation of central configuration (\ref{EquacaoDeConfiguracaoCentral}) associated to a potential like in (\ref{PotencialQualquer}). 
We can choose $r_1,r_2\ll r_3,r_4 \ll \cdots \ll r_{L-2},r_{L-1}\ll r_L $ in such way 
that there is a relative equilibrium where the configuration assume the shape of $L$ homothetic regular polygons with  these radii, moreover such equilibrium only occurs  if the masses in each polygon are equal.\end{theorem}

\begin{proof}
The demonstration is analogous to that in the previous theorem. To prove the first case of induction where $L=1$, we analyze the cases where $a=3$ e $a=2$. In this case the determinant corresponds to $\xi_p(r)$. By homogeneity, is sufficient to analyze the case $r=1$.
For $a=3$, note that $\xi_{k-1}(1)$ corresponds to a expression for $\lambda_k$ in the corollary of lemma $9$ in \cite{13}, by lemma $12$ in  \cite{13}, such expression is different from zero for $k\neq N$ and $k\neq \frac{N+1}{2}$.

This we conclude that $\xi_{p}$ is different from zero for $p\neq N-1$ and $p \neq \frac{N-1}{2}.$
For $a=2$, we get:
\begin{align*}
\xi_p(1)&=&\frac{1}{2}\sum_{j=1}^{N-1} \frac{\sin\left(\frac{\pi j (2p+1)}{N}\right)}{\sin\left(\frac{\pi j}{N}\right)}=\frac{1}{2}\sum_{j=1}^{N-1} \frac{e^{\frac{\pi j (2p+1)}{N}\sqrt{-1}}-e^{-\frac{\pi j (2p+1)}{N}\sqrt{-1}}}{e^{\frac{\pi j}{N}\sqrt{-1}}-e^{-\frac{\pi j }{N}\sqrt{-1}}}\\
&=& \frac{1}{2}\sum_{j=1}^{N-1} \sum_{k=-p}^{p}e^{\frac{2\pi jk}{N}\sqrt{-1}}=\frac{1}{2}\sum_{k=-p}^{p}\sum_{j=1}^{N-1}  e^{\frac{2\pi jk}{N}\sqrt{-1}}=\frac{N-(2p+1)}{2}.
\end{align*}
To see the last equality, it is suffices note that $\sum_{j=1}^{N-1}  e^{\frac{2\pi jk}{N}\sqrt{-1}}$ is $N-1$ or $-1$ accordingly $k$ is a multiple of $N$ or not.
This proofs that $\xi_p(1)$ is different from zero for $p \neq \frac{N-1}{2}$.
The theorem follows.
\end{proof}
\begin{remark}
Again, is worth to remember that vector $a\mathbf{v}_{N-1}+b\mathbf{v}_{\frac{N-1}{2}}\in \mathbb{R}^N$ then $N=3$ or $a=b=0$. To see demonstration consult lemma $7$ in \cite{13}.

By the evidences, in the cases with $L \leq 2$ or in the general case for very different radii we are led to believe in the following conjecture.
\end{remark}
\begin{conj}
A central configuration formed by homothetic polygons is only possible if the masses in each polygon are equal.
\end{conj}
\section{The non-planar case}
\label{sec:5}
We can consider, with some adjusts according follows, the case where the polygons do not lie in the same plane. In this case for $a=2$, the equations \eqref{EQUACAODECONFIGURACAOCENTRAL}, no longer represent the vortex dynamics, but we still have interest in the general case, which include this one. 
The positions are given by 
\begin{equation}
q_{(j-1)N+k}=(r_j\omega_k,h_j) \in \mathbb{C}\times\mathbb{R}\label{RelacoesEntrePosicoesCasoNaoPlanar},\mbox{ for } \quad j\in I_L \mbox{ and } k\in I_N.
\end{equation}
where $h_j\in \mathbb{R}$, and $h_j\neq h_i$ if $r_j=r_i$.
In this case, the equation \eqref{EQUACAODECONFIGURACAOCENTRAL}, becomes a  vectorial equation, which is equivalent to two scalar systems, one for each coordinate. Such systems are written in matricial form like
\begin{equation}
 \left[\begin{array}{cccc}
  A_{11}&A_{12}&\ldots&A_{1L}\\
  A_{21}&A_{22}&\ldots&A_{2L}\\
  \vdots&\vdots&\ddots&\vdots\\
  A_{L1}&A_{L2}&\ldots&A_{LL}\\
  \end{array}\right]
   \left[\begin{array}{c}
  \mathbf m_{1}\\
  \mathbf m_{2}\\
    \vdots\\
 \mathbf m_{L}
  \end{array}\right]=\left[\begin{array}{c}
  \nu^2r_1 \mathbf{1}\\
  \nu^2r_2\mathbf{1}\\
    \vdots\\
 \nu^2r_L\mathbf{1}\\
  \end{array}\right]\label{EquacaoMatricialParaOProblemaCasoNaoPlanarA},
  \end{equation}

   and
  \begin{equation}
 \left[\begin{array}{cccc}
  B_{11}&B_{12}&\ldots&B_{1L}\\
  B_{21}&B_{22}&\ldots&B_{2L}\\
  \vdots&\vdots&\ddots&\vdots\\
  B_{L1}&B_{L2}&\ldots&B_{LL}\\
  \end{array}\right]
   \left[\begin{array}{c}
  \mathbf{m}_{1}\\
  \mathbf{m}_{2}\\
    \vdots\\
 \mathbf{m}_{L}
  \end{array}\right]=\left[\begin{array}{c}
  \nu^2h_1 \mathbf{1}\\
  \nu^2h_2\mathbf{1}\\
    \vdots\\
 \nu^2h_L\mathbf{1}\\
  \end{array}\right]
  \label{EquacaoMatricialParaOProblemaCasoNaoPlanarB},
  \end{equation}
  where the matrix $A_{TS}=[a_{kj}]$ has entries
  \begin{equation}
 a_{kj}=\left\lbrace \begin{array}{ccc}
  \dfrac{r_T-r_S\omega_{j-k}}{|(r_T,h_T)-(r_S\omega_{j-k},h_S)|^{a}}+\frac{\nu^2}{M}r_S\omega_{j-k}& \text{for} & k,j\in I_{N}\label{MatrizATSCasoNaoPlanarA};\\
  a_{kk}=\frac{\nu^2}{M}r_T& \text{for} & T=S
  \end{array}\right.
  \end{equation}
And $B_{TS}=[b_{kj}]$
  \begin{equation}
 b_{kj}=  \left\lbrace\begin{array}{ccc}
 \dfrac{h_T-h_S}{|(r_T,h_T)-(r_S\omega_{j-k},h_S)|^{a}}+\frac{\nu^2}{M}h_S& \text{for} & k,j\in I_{N}\label{MatrizATSCasoNaoPlanarB};\\
  b_{kk}=\frac{\nu^2}{M}h_T& \text{for} & T=S.
  \end{array}\right.
  \end{equation}

After the reduction to subsystems, we get similar systems to those in \eqref{SubsistemaP},\,\eqref{SubsistemaN}. However the coefficients now are given by
  \begin{align}
\left\{\begin{array}{ccl}
\lambda_p(A_{TS})&=&\displaystyle \sum_{j=1}^{N} \dfrac{r_T-r_S\omega_{j-1}}{|(r_T,h_T)-(r_S\omega_{j-1},h_S)|^{a}}\omega_{p}^{j-1}+\delta_{p,N-1}r_S\frac{\nu^2}{M}N,\quad \mbox{ if, } T \neq S,\\
\lambda_p(A_{TT})&=&\displaystyle \sum_{j=2}^{N} \dfrac{r_T-r_T\omega_{j-1}}{|(r_T,h_T)-(r_T\omega_{j-1},h_T)|^a}\omega_{p}^{j-1}+\delta_{p,N-1}r_T\frac{\nu^2}{M}N.
\end{array}
\right.\label{WeAreTHeChildren}
\end{align}

And
\begin{align}
\left\{\begin{array}{ccl}
\lambda_p(B_{TS})&=&\displaystyle \sum_{j=1}^{N} \dfrac{h_T-h_S}{|(r_T,h_T)-(r_S\omega_{j-1},h_S)|^{a}}\omega_{p}^{j-1}+\delta_{p,N}h_S\frac{\nu^2}{M}N,\quad \mbox{ if, } T \neq S,\\
\lambda_p(B_{TT})&=&\delta_{p,N}h_T\frac{\nu^2}{M}N.
\end{array}
\right.\label{WeAreTheWorld}
\end{align}
So the masses must satisfy simultaneously both systems whose coefficients are 
$\lambda_p(A_{TS})$ and $\lambda_p(B_{TS})$.

\subsection{The case of two polygons}

It is interesting to note that by Theorem $2$ of \cite{5}, the central configuration with two planar regular polygons in different planes, in fact, exists when $a=3$.
 We show in the sequel that for any value of $a\geq 2$, such configurations only could occur when the masses in each polygon are equal (for $a=3$ this fact is demonstrated in \cite{2015}).

In the case of  two polygons we have $\det[\lambda_p(A_{TS})]$ given explicitly by
\begin{align*}
\left|\begin{array}{cc}
\sum_{j=1}^{N-1} \frac{r_1\cos\left( j\theta p\right)-r_1\cos\left( j\theta(p+1)\right)}{\left(r_1^2-2r_1^2\cos\left( j\theta \right)+r_1^2+(h_1-h_1)^2\right)^{\frac{a}{2}}}&\sum_{j=1}^{N} \frac{r_1\cos\left( j\theta p\right)-r_2\cos\left( j\theta (p+1)\right)}{\left(r_1^2-2r_1r_2\cos\left( j\theta\right)+r_2^2+(h_1-h_2)^2\right)^{\frac{a}{2}}} \\ 
\sum_{j=1}^{N} \frac{r_2\cos\left( j\theta p\right)-r_1\cos\left( j\theta (p+1)\right)}{\left(r_1^2-2r_1r_2\cos\left( j\theta\right)+r_2^2+(h_1-h_2)^2\right)^{\frac{a}{2}}}& \sum_{j=1}^{N-1} \frac{r_2\cos\left( j\theta p\right)-r_2\cos\left( j \theta(p+1)\right)}{\left(r_2^2-2r_2^2\cos\left( j\theta\right)+r_2^2+(h_2-h_2)^2\right)^{\frac{a}{2}}}
\end{array} \right|
\end{align*}

Unlike the planar case, the terms of secondary diagonal in the matrix of coefficients do not always have opposite signs. For some fixed values of $p$, the signs may be either equal or different according change the radii $r_1,r_2$ and the heights $h_1,h_2$. 
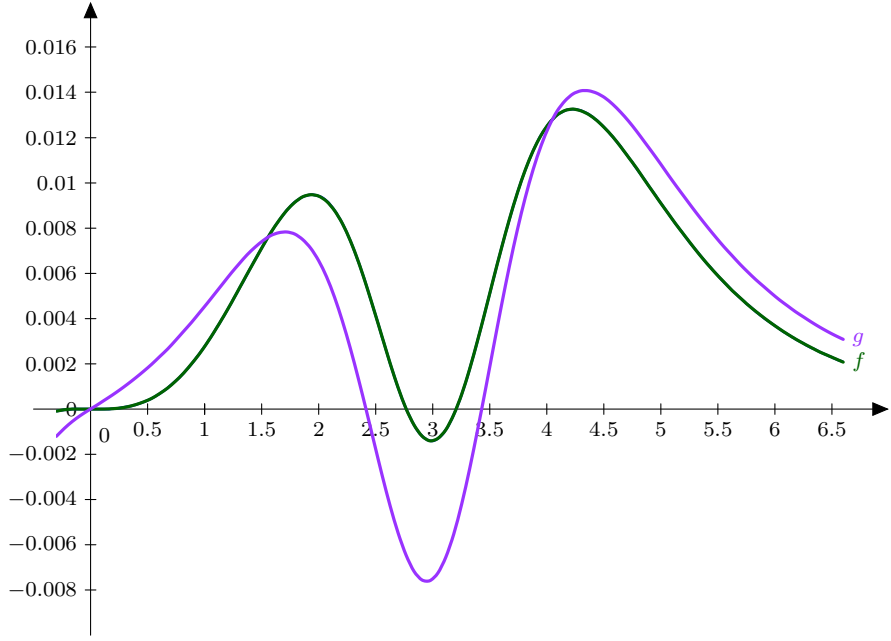
\begin{figure} 
\begin{tikzpicture}[line cap=round,line join=round,>=triangle 45,x=1.5cm,y=300.0cm]
\draw[->,color=black] (-0.5,0.) -- (7.0,0.);
\foreach \x in {,0.5,1,1.5,2,2.5,3,3.5,4,4.5,5,5.5,6,6.5}
\draw[shift={(\x,0)},color=black] (0pt,2pt) -- (0pt,-2pt) node[below] {\footnotesize $\x$};
\draw[->,color=black] (0.,-0.01) -- (0.,0.018);
\foreach \y in {-0.008,-0.006,-0.004,-0.002,0,0.002,0.004,0.006,0.008,0.01,0.012,0.014,0.016}
\draw[shift={(0,\y)},color=black] (2pt,0pt) -- (-2pt,0pt) node[left] {\footnotesize $\y$};
\draw[color=black] (0pt,-10pt) node[right] {\footnotesize $0$};
\clip(-0.3,-0.01) rectangle (8.18,0.016);
\draw[line width=1pt,smooth, color=blue, domain=-0.3:6.6, line cap=butt, samples=400]    plot (\x,{(-(6.0/(193.0/16.0+\x^(2.0))^(3.0/2.0))-(-3.0+\x)/(193.0/16.0-6.0*\x+\x^(2.0))^(3.0/2.0)-(-3.0-\x)/(193.0/16.0+6.0*\x+\x^(2.0))^(3.0/2.0))*((2.0*\x)/(193.0/16.0+\x^(2.0))^(3.0/2.0)+(3.0-\x)/(193.0/16.0-6.0*\x+\x^(2.0))^(3.0/2.0)+(-3.0-\x)/(193.0/16.0+6.0*\x+\x^(2.0))^(3.0/2.0))});
\draw[line width=1.2pt,color=qqwuqq,smooth,samples=100,domain=-0.3:6.6] plot(\x,{(-(6.0/(193.0/16.0+\x^(2.0))^(3.0/2.0))-(-3.0+\x)/(193.0/16.0-6.0*\x+\x^(2.0))^(3.0/2.0)-(-3.0-\x)/(193.0/16.0+6.0*\x+\x^(2.0))^(3.0/2.0))*((2.0*\x)/(193.0/16.0+\x^(2.0))^(3.0/2.0)+(3.0-\x)/(193.0/16.0-6.0*\x+\x^(2.0))^(3.0/2.0)+(-3.0-\x)/(193.0/16.0+6.0*\x+\x^(2.0))^(3.0/2.0))}) node[right] {$f$}; 
\draw[line width=1.2pt,color=zzttff,smooth,samples=100,domain=-0.3:6.6] plot(\x,{(6.0/(193.0/16.0+\x^(2.0))^(3.0/2.0)-(-3.0+\x)/(193.0/16.0-6.0*\x+\x^(2.0))^(3.0/2.0)-(-3.0-\x)/(193.0/16.0+6.0*\x+\x^(2.0))^(3.0/2.0))*(-((2.0*\x)/(193.0/16.0+\x^(2.0))^(3.0/2.0))+(3.0-\x)/(193.0/16.0-6.0*\x+\x^(2.0))^(3.0/2.0)+(-3.0-\x)/(193.0/16.0+6.0*\x+\x^(2.0))^(3.0/2.0))}) node[right] {$g$};
\end{tikzpicture}
\caption{Product in Secondary Diagonal in the Non-planar Case}
\end{figure}
 In the Figure $1$, $f$ and $g$ represents respectively the products $-\lambda_1(A_{12})*\lambda_1(A_{21})$ and $-\lambda_3(A_{12})*\lambda_3(A_{21})$ in function of $r_1$ to $(r_2,h_1,h_2,a,N)=(3,\frac{1}{4},2,3,4)$. 
 So the analysis in this case is much more complicated, and we do not obtain the same results. 
 However, by the existence of other system, the analysis turn out to be more simple. In fact, if $p\neq N$ by equation \eqref{WeAreTheWorld}, we have $\lambda_p(B_{11})=\lambda_p(B_{22})=0,$ so
\begin{align*}
&\det[\lambda_p(B_{TS})]=\left|\begin{array}{cc}
\lambda_p(B_{11})&\lambda_p(B_{12})\\ 
\lambda_p(B_{21})&\lambda_p(B_{22})
\end{array} \right|=\\
&\left|\begin{array}{cc}
 0&\sum_{j=1}^{N} \frac{(h_1-h_2)\cos\left( j\theta p\right)}{\left(r_1^2-2r_1r_2\cos\left( j\theta\right)+r_2^2+(h_1-h_2)^2\right)^{\frac{a}{2}}} \\ 
\sum_{j=1}^{N} \frac{(h_2-h_1)\cos\left( j\theta p\right)}{\left(r_1^2-2r_1r_2\cos\left(j\theta\right)+r_2^2+(h_1-h_2)^2\right)^{\frac{a}{2}}}& 0
\end{array} \right|
\end{align*}

Consider that $h_1\neq h_2$, is sufficient to show that the expression 
\begin{align*}
\sum_{j=1}^{N} \frac{\cos\left( j\theta p\right)}{\left(r_1^2-2r_1r_2\cos\left( j\theta\right)+r_2^2+(h_1-h_2)^2\right)^{\frac{a}{2}}}
=\\
\frac{1}{r_1^a}\sum_{j=1}^{N} \frac{e^{\left( j p\theta\right)\sqrt{-1}}}{\left(1-2\frac{r_2}{r_1}\cos\left( j\theta\right)+\left(\frac{r_2}{r_1}\right)^2+\left(\frac{h_2-h_1}{r_1}\right)^2\right)^{\frac{a}{2}}}
\end{align*} is non-zero. Without loss of generality we can assume that $r_1\geq r_2$. Setting $x=\frac{r_2}{r_1}$ and $k=\sqrt{\frac{|h_2-h_1|}{r_1}}$. This expression becomes
$\frac{1}{r_1^a}\sum_{j=1}^{N} \frac{e^{\left( j \theta p\right)\sqrt{-1}}}{\left(1-2x\cos\left(\theta j\right)+x^2+k\right)^{\frac{a}{2}}}
$,
by \cite{2015} (corollary $2.3$) such expression is always greater than zero. So the determinant is nonzero  which implies that for a
mass vector like \eqref{VetorDeMassas}, the coefficients $x_1^p,x_2^p$ for $p \in I_{N-1}$ are zero. So if the central configuration exists, the masses in each polygon have to be equal.
Now to explore the existence, we analyze the case $p=N$. Using \eqref{WeAreTheWorld} the expression for $\det[\lambda_N(B_{TS})]$ is given by 
\begin{align*}
\left|\begin{array}{cc}
 h_1\frac{\nu^2}{M}N&\psi+h_1\frac{\nu^2}{M}N\\ 
\psi+h_2\frac{\nu^2}{M}N& h_2\frac{\nu^2}{M}N
\end{array} \right|,
\end{align*}
where $\psi=\sum_{j=1}^{N} \frac{(h_2-h_1)}{\left(r_1^2-2r_1r_2\cos\left(\theta j\right)+r_2^2+(h_1-h_2)^2\right)^{\frac{a}{2}}}$.

Without loss of generality, choosing $h_1=0$ this expression reduces to\\
$ h_2^2\cdot\tilde{\psi}\cdot(\tilde{\psi}+\frac{\nu^2}{M}N)$
where $\tilde{\psi}=\sum_{j=1}^{N} \frac{1}{\left(r_1^2-2r_1r_2\cos\left( j\theta\right)+r_2^2+(-h_2)^2\right)^{\frac{a}{2}}}$.
This determinant is obviously positive (provided that $\nu$ is real). This shows that the system have a solution (indeed this is a necessary condition to existence of the central configuration,  but not sufficient, once the system for the coefficients in \eqref{WeAreTHeChildren} has to be satisfied as well).

\subsection{The case of more than two polygons}
The matrix $[\lambda_p(B_{TS})]$ is skew-symmetric if $p\neq N$, because \\$\lambda_p(B_{TS})=\sum_{j=1}^{N} \frac{(h_T-h_S)\cos\left( j\theta p\right)}{\left(r_T^2-2r_Sr_T\cos\left(j\theta \right)+r_S^2+(h_T-h_S)^2\right)^{\frac{a}{2}}}$, therefore its determinant is zero if its order (the number of polygons) is odd.\\
In those cases inevitably we are led to study  the determinants $\det[\lambda_p(A_{TS})]$ that as we have seen, are more difficult
than the planar case, which make the full analysis very complicated for more than two polygons. If the number of polygons is even, is possible to obtain partial results for the determinant $[\lambda_p(B_{TS})]$, keeping fixed the radii and change the heights by repeat the argument of Theorem \ref{EvenPolygons}.

\begin{acknowledgements}
The author would like to thank Eduardo S. Leandro for being the advisor on this work. And to Thiago Dias by the helpful comments.
And the Department of Mathematics at {\it Universidade Federal Rural de Pernambuco} for their assistance.
\end{acknowledgements}



\end{document}